\documentclass[11pt]{amsart}						
\usepackage[pagebackref,  pdfpagemode=FullScreen,  colorlinks=true]{hyperref}	
\usepackage{graphicx}
\usepackage{verbatim}
\usepackage[all,cmtip]{xy}
\usepackage[english]{babel}

	\usepackage{color}

	\newtheorem{thm}{Theorem}
  	\newtheorem{cor}{Corollary}
  	\newtheorem{lem}{Lemma}
  	\newtheorem{prop}{Proposition}
  	\newtheorem{prob}{Problem}	

	\theoremstyle{definition}
  	\newtheorem{df}{Definition}
	
	\theoremstyle{remark}
  	\newtheorem{rem}{Remark}

	\makeatletter
	\@namedef{subjclassname@2010}{%
	\textup{2010} Mathematics Subject Classification}
	\makeatother

\def\PP{{\textbf P}}
\def\OO{\mathcal{O}}

\def\W{\mathcal{W}}

\def\cM{\mathcal{M}}

\def\cC{\mathcal{C}}

\def\Pic0{{\rm Pic}^0(X)}

\def\ww{\overline{\mathcal{W}}}

\def\mm{\overline{\mathcal{M}}}

\def\cc{\overline{\mathcal{C}}}

\begin{document}

\title{Du Val curves and the pointed Brill-Noether Theorem}

\author[G. Farkas]{Gavril Farkas}

\address{Humboldt-Universit\"at zu Berlin \ Institut f\"ur Mathematik
\hfill \newline\texttt{}
 \indent  Unter den Linden 6, 10099 Berlin, Germany} \email{{\tt farkas@math.hu-berlin.de}}
\thanks{}

\author[N. Tarasca]{Nicola Tarasca}
\address{University of Utah \ Department of Mathematics  \hfill
 \newline \indent Salt Lake City,  UT 84112, USA }
 \email{{\tt
tarasca@math.utah.edu}}


\begin{abstract} We show that a general curve in an explicit class of what we call Du Val pointed curves  satisfies the Brill-Noether Theorem for pointed curves.
Furthermore, we prove that a generic pencil of Du Val pointed curves is disjoint from all Brill-Noether divisors on the universal curve. This provides explicit examples of smooth pointed curves of arbitrary genus defined over $\mathbb Q$ which are Brill-Noether general. A similar result is proved for $2$-pointed curves as well using explicit curves on elliptic ruled surfaces.
\end{abstract}

\maketitle

The pointed Brill-Noether Theorem concerns the study of linear series on a general pointed algebraic curve $[C,p]$ with prescribed ramification at the marked point $p$.
Recall that for a point $p\in C$ and a linear series $\ell=(L, V)\in G^r_d(C)$, one denotes by $$\alpha^{\ell}(p):0\leq \alpha_0^{\ell}(p)\leq \ldots \leq \alpha_r^{\ell}(p)\leq d-r$$
the \emph{ramification} sequence of $\ell$ at $p$. One says that $p\in C$ is a ramification point of $\ell$ if $\alpha^{\ell}_r(p)>0$. For instance, the ramification points of the canonical linear series are precisely the Weierstrass points of $C$. The total number of ramification points of $\ell$, counted with appropriate multiplicities, is given by the \emph{Pl\"ucker formula}, see for instance \cite{MR846932} Proposition 1.1.  Fixing a Schubert index $\alpha:0\leq \alpha_0\leq \ldots \leq \alpha_r\leq d-r$, one can ask when a general pointed curve $[C,p]$ of genus $g$ carries a linear series $\ell\in G^r_d(C)$ with ramification sequence $\alpha^{\ell}(p)\geq \alpha$. The locus $G^r_d(C, p, \alpha)$ of linear series on $C$ satisfying this condition is a generalized determinantal variety of expected dimension
$$
\rho(g, r, d, \alpha):=\rho(g, r, d)-w(\alpha),
$$ where $\rho(g, r, d):=g-(r+1)(g-d+r)$ and $w(\alpha):=\alpha_0+\cdots+\alpha_r$ is the weight of $\alpha$.  It is proved in \cite{MR910206} Theorem 1.1 that for a general pointed curve $[C, p]\in \cM_{g, 1}$, each component of $G^r_d(C, p, \alpha)$, if nonempty, has dimension precisely $\rho(g, r, d, \alpha)$.
Moreover, \cite{MR910206} Proposition 1.2 establishes that $G^r_d(C,p, \alpha)\neq \emptyset$ if and only if
$$
\sum_{i=0}^r \mbox{max}\{\alpha_i+g-d+r,0\}\leq g.
$$
The proofs in \cite{MR910206} rely on limit linear series and degeneration to the boundary of the universal curve $\cc_g:=\mm_{g,1}$. Up to now, no  examples whatsoever of \emph{smooth pointed} curves $[C,p]\in \cC_{g}$ verifying the pointed Brill-Noether Theorem have been known. This situation contrasts the classical Brill-Noether Theorem; even though the original proof in \cite{MR563378} used degeneration to nodal curves, soon afterwards, in his well-known paper \cite{MR852158}, Lazarsfeld showed that sections of general polarized $K3$ surfaces are Brill-Noether-Petri general. Since curves in the polarization class of a $K3$ surface have no obvious distinguished marked points, it is far from clear how to extend the results of \cite{MR852158} to the case of pointed curves. In  \cite{ABFS}, an \emph{explicit} specialization of Lazarsfeld's curves emerging from the paper  \cite{ABS-W} is worked out.  It is shown that suitably general singular plane curves of degree $3g$ having multiplicity $g$ at eight points in $\PP^2$ and multiplicity $g-1$ at a further ninth point verify the Brill-Noether-Petri Theorem. Such  curves, which belong to the closure in $\mm_g$ of the locus of curves lying on $K3$ surfaces,  are called \emph{Du~Val curves} of genus $g$.

\vskip 4pt

One aim of this paper is to show that the Du Val curves introduced in \cite{ABFS}  lead to Brill-Noether general smooth \emph{pointed} curves of any genus defined over $\mathbb Q$.
The essential observation is that, unlike curves on general $K3$ surfaces, Du Val curves have a distinguished marked point with respect to which a pointed Brill-Noether Theorem can be established.

\vskip 5pt

We begin by recalling the setting of \cite{ABFS}. Let $S'$ be the blow-up of $ \PP^2$ at nine points
 $p_1,\dots, p_{9}$ which are \emph{general} in the sense of \cite{ABFS} (see also Section \ref{sect1} for the precise definition).
Let $E_1,\dots, E_{9}$ be the exceptional curves on $S'$. We denote by
$
J'\in |-K_{S'}|
$ the unique smooth plane cubic passing through $p_1, \ldots, p_9$ and consider the linear system on $S'$
$$
L_g:=\bigl|3g\ell-gE_1-\cdots-gE_{8}-(g-1)E_{9}\bigr|,
$$
where $\ell\in \mbox{Pic}(S')$ is the proper transform of a line in $\PP^2$. The main result of \cite{ABFS} is that a general curve $C'\in L_g$ verifies the Brill-Noether-Petri Theorem. For each $g\geq 1$, the points $p_1, \ldots, p_9$ determine a $10$-th point $p_{10}=p_{10}^{(g)}$ which is the base point of $L_g$. In fact, $p_{10} \in C'\cdot J'$, for every $C'\in L_g$. The point $p_{10}$ is determined by the relation
\begin{equation}\label{p10}
p_{10}=p_{10}^{(g)}= -gp_1-\cdots-gp_8-(g-1)p_9\in J',
\end{equation}
with respect to the group law of the elliptic curve. Under the genericity assumptions on the points $p_1, \ldots, p_9$ we started with, the points $p_{10}^{(g)}$ are distinct from one another, as well as from $p_1, \ldots, p_9$, see also Proposition \ref{notorsion}. As in \cite{ABFS}, we set $S:=\mbox{Bl}_{p_{10}}(S)$ and, by slight abuse of notation, we denote by $E_1, \ldots, E_{10}$ the corresponding exceptional curves. If $C$ is the strict transform of $C'$, then $|\OO_S(C)|$ is a base point free linear system of curves of genus $g$ having a section induced by $E_{10}$ (note that $C\cdot E_{10}=1$).

\vskip 4pt

A \emph{pointed Du Val curve} is a smooth  pointed  curve $[C, p]\in \cc_{g}$, where $C\subset S$ is as above and $\{p\}=C\cdot E_{10}$.
Before stating our main results, we recall that for a linear system $\ell\in G^r_d(C)$ and points $p_1, \ldots, p_n\in C$, the \emph{pointed Brill-Noether number} is defined as $$\rho(\ell,p_1, \ldots, p_n):=\rho(g,r,d)-w\bigl(\alpha^{\ell}(p_1)\bigr)-\cdots-w\bigl(\alpha^{\ell}(p_n)\bigr).$$

\begin{thm}\label{main1}
A general pointed Du Val curve $[C,p]$ verifies the pointed Brill-Noether Theorem, that is, $\dim G^r_d(C,p, \alpha)=\rho(g,r, d,\alpha)$, when $G^r_d(C,p, \alpha)\neq \emptyset$. In particular, for every linear system $\ell$ on $C$, one has $\rho(\ell,p)\geq 0$.
\end{thm}

Since the points $p_1, \ldots, p_9$ can be chosen to have rational coefficients, $p=p_{10}^{(g)}\in \PP^2(\mathbb Q)$ and then $[C,p]$ is also defined over $\mathbb Q$. Hence, paralleling \cite{ABFS} Corollary 1.3,  our Theorem \ref{main1} provides examples of Brill-Noether general pointed curves of arbitrary genus $g$ defined over $\mathbb Q$.

\vskip 4pt

If $\W_g$ denotes the locus of Weierstrass points in $\cC_g$ (known to be an irreducible divisor on the universal curve), by direct calculation we show that the image of the family $j:\PP^1\rightarrow \cc_g$ induced by a Lefschetz
pencil of Du Val curves on $S$ satisfies
$$
j(\PP^1)\cap \ww_g=\emptyset,
$$
that is, for \emph{every} pointed Du Val curve $[C, p]$, the marked point $p$ is not a Weierstrass point of $C$. As we point out in Corollary \ref{ptp},  this implies that $j(\PP^1)$ is disjoint from all pointed Brill-Noether divisors on $\cc_g$. We refer to Section \ref{sect1} for detailed background on pointed Brill-Noether divisors on $\cc_g$.

\vskip 5pt

\subsection{Brill-Noether general $2$-pointed curves on elliptic ruled surfaces.} The Brill-Noether problem can be formulated for $n$-pointed curves $[C,p_1, \ldots, p_n]$ and concerns the variety of linear series $\ell\in G^r_d(C)$ having prescribed ramification
$\alpha^{\ell}(p_i)\geq \alpha^i$ for $i=1, \ldots, n$, given in terms of fixed Schubert indices $\alpha^1, \ldots, \alpha^n$.
In Section \ref{sect2}, using \emph{decomposable} elliptic ruled surfaces, we exhibit for the first time examples of  smooth $2$-pointed curves of arbitrary genus verifying the $2$-pointed Brill-Noether Theorem. The construction is inspired by a very nice note of Treibich \cite{MR1218267}.

\vskip 4pt

We start with an elliptic curve $J$ and a non-torsion line bundle $\eta\in \mbox{Pic}^0(J)$. The decomposable ruled surface
$$
\phi:Y:=\PP(\OO_J\oplus \eta)\rightarrow J
$$
is endowed with two disjoint  sections $J_0$ and $J_1$ respectively. Pick a point $r\in J$ and denote by $f:=\phi^{-1}(r)$ the corresponding ruling of $Y$. We denote by $s=s^{(g)}\in J$ the point determined by the equation $\OO_J(s-r)=\eta^{\otimes g}$. The linear system $|gJ_0+f|$ consists of curves of genus $g$ and has two base points, namely
$$\{p\}:= \phi^{-1}(r)\cdot J_1 \ \mbox{ and } \ \{q\}:= \phi^{-1}(s)\cdot J_0,$$
respectively. We establish the following result:

\begin{thm}\label{main3}
The $2$-pointed curve $[C,p,q]\in \cM_{g,2}$, where $C \in |gJ_0+f|$ is a general element and $p$ and $q$ are as above, verifies the $2$-pointed Brill-Noether Theorem.
In particular, for every linear series $\ell\in G^r_d(C)$ the inequality $\rho(\ell, p,q)\geq 0$ holds.
\end{thm}

\vskip 4pt

A Brill-Noether general $2$-pointed curve  supports a Brill-Noether general $1$-pointed curve obtained by dropping either marked point. In particular, both $1$-pointed curves $[C,p]$ and $[C,q]$ in the statement of Theorem \ref{main3} verify the $1$-pointed Brill-Noether Theorem as well. For details, we refer to Section \ref{sect2}. The proofs of both Theorems \ref{main1} and \ref{main3} are intimately related, and rely on a canonical degeneration within the corresponding linear system on the surface to a singular curve with an elliptic tail. This leads to an inductive argument in the genus, which ultimately proves the desired Brill-Noether type theorems.

\vskip 4pt

Arguably, for many applications, the curves constructed in Theorem \ref{main3} are the simplest known examples of Brill-Noether general smooth curves of arbitrary genus. They combine two desirable features: (i) The canonical elliptic tail degeneration in $|gJ_0+f|$ provides a system of Brill-Noether general curves of any genus on the surface $Y$, which invites inductive proofs and reduction to genus $1$ curves and Schubert calculus problems in the spirit of limit linear series, and (ii) The general curve in $|gJ_0+f|$ being smooth, one need not build-up the degeneration set-up typical for limit linear series applications. A vivid instance of their use  is the recent proof in \cite{FK1} of the Prym-Green Conjecture concerning the naturality of the resolution of a paracanonical curve $\varphi_{K_C\otimes \eta}:C\hookrightarrow \PP^{g-2}$, where $C$ is a general curve of odd genus and $\eta$ is an $\ell$-torsion line bundles on $C$. The conjecture is proven for odd $g$ and arbitrary $\ell$ using precisely the curves constructed in Theorem \ref{main3}. For a proof of the Prym-Green Conjecture  using special $K3$ surfaces instead --- but only in the range $\ell\geq \sqrt{\frac{g+2}{2}}$ --- see \cite{FK2}.

\vskip 3pt

Finally, we show in Theorem \ref{main2} that Brill-Noether general one-pointed smooth curves can be constructed also on indecomposable elliptic ruled surfaces.

\vskip 4pt

\noindent {\bf Acknowledgments:} Section \ref{sect1} of this paper  uses in an essential way the methods developed in \cite{ABFS}. The first author is grateful to Enrico Arbarello for interesting discussions related to this circle of ideas. The presentation of the paper clearly benefitted from the insightful remarks of two referees, whom we thank.

\section{Pointed Du Val curves and  Weierstrass points}\label{sect1}

We assume familiarity with the theory of limit linear series in the sense of \cite{MR846932}. We need a few facts concerning divisor classes on the universal curve $\cc_g:=\mm_{g,1}$. The rational Picard group $\mbox{Pic}(\cc_{g})$ is generated by the Hodge class $\lambda$, the relative cotangent class $\psi$, the boundary divisor class $\delta_{\mathrm{irr}}:=[\Delta_{\mathrm{irr}}]$ of irreducible pointed stable
curves of genus $g$ and by the classes $\delta_{i}:=[\Delta_i]$, where for each $i=1, \ldots, g-1$,  the boundary divisor $\Delta_{i}\subset \cc_g$ corresponds to a transverse union of two smooth curves of genus $i$ and $g-i$ respectively, meeting in one point,  the marked points lying on the genus $i$ component. If $\pi:\cc_g\rightarrow \mm_g$ is the morphism forgetting the marked point,   the boundary divisors on $\mm_g$ and those on $\cc_g$ are related by the following formulas:
\[
\pi^*(\delta_{\mathrm{irr}})=\delta_{\mathrm{irr}},  \mbox{ }    \pi^*(\delta_i)=\delta_i+\delta_{g-i}, \mbox{ for } 1\leq i < \frac{g}{2}, \mbox{ and } \pi^*(\delta_{\frac{g}{2}})= \delta_{\frac{g}{2}}, \mbox{ for $g$ even}.
\]

\vskip 3pt

If $\alpha:0\leq \alpha_0\leq \ldots \leq \alpha_r\leq d-r$ is a Schubert index of type $(r,d)$, we introduce the \emph{complementary} Schubert index  $\alpha^c:0\leq d-r-\alpha_r\leq \ldots \leq d-r-\alpha_0\leq d-r$. When $\alpha_i=0$ for $i=0, \ldots, r$, we say that $\alpha$ is the trivial Schubert index.  We recall the definition of \emph{pointed Brill-Noether divisors} on $\cc_g$. Fix integers $r,d\geq 1$ and a Schubert index $\alpha:0\leq \alpha_0\leq \ldots \leq \alpha_r\leq d-r$ such that
the expected dimension of the locus of linear series of type $\mathfrak g^r_d$ on a curve of genus $g$ with prescribed ramification $\alpha$ at a given point equals $-1$. In other words,
$$\rho(g,r,d, \alpha):=g-(r+1)(g-d+r)-w(\alpha)=-1.$$
Let $\cC_{g,d}^r(\alpha):=\bigl\{[C,p]\in \cC_g: G^r_d(C,p,\alpha)\neq \emptyset\bigr\}$ be the corresponding pointed Brill-Noether locus. For instance,
$$
\W_g:=\cC_{g,2g-2}^{g-1}(0,\ldots, 0,1)=\Bigl\{[C,p]\in \cC_g: H^0(C,\omega_C(-gp))\neq 0\Bigr\}
$$
is the divisor of Weierstrass points. Since $\W_g$ can be parametrized by the Hurwitz space of $g$-fold covers of $\PP^1$ having a point of total ramification, it follows from \cite{Arb} Theorem~2.5 that $\W_g$ is an irreducible divisor. If $\rho(g,r,d)=-1$, then $\cC_{g,d}^r(0, \ldots, 0)$ is the pull-back to $\cC_g$ of the Brill-Noether divisor $\cM_{g,d}^r$ consisting of curves carrying a $\mathfrak g^r_d$.

\vskip 3pt
Cukierman \cite{MR1016424} computed the class of the closure $\ww_g$ of the Weierstrass divisor in $\cc_g$:
\begin{equation}\label{wei}
[\ww_g]=-\lambda+{g+1\choose 2} \psi-\sum_{i=1}^{g-1} {g-i+1\choose 2}\delta_i\in \mbox{Pic}(\cc_g).
\end{equation}
We also recall \cite{MR910206} that the class of the pull-back to $\cc_g$ of the Brill-Noether divisors $\mm_{g,d}^r$ is given by the formula
\begin{equation}\label{bn1}
[\overline{\mathcal{C}}_{g,d}^r(0,\ldots, 0)]=c_{g,d,r}\cdot \mathcal{BN}_g,
\end{equation}
where $c_{g,d,r}\in \mathbb Q_{>0}$ and $$\mathcal{BN}_g:=(g+3)\lambda-\frac{g+1}{6}\delta_{\mathrm{irr}}-\sum_{i=1}^{g-1} i(g-i)\delta_i\in \mathrm{Pic}(\overline{\mathcal{C}}_g).$$

Remarkably, the pointed Brill-Noether divisors only span a $2$-dimensional cone in $\mbox{Pic}(\cc_g)$. It is shown in \cite{MR985853} Theorem 1.2 that $\cc_{g,d}^r(\alpha)$ is a proper subvariety of $\cc_g$, having a unique divisorial component. The class of this component, which we shall denote by $[\cc_{g,d}^r(\alpha)]\in CH^1(\cc_g)$, can be written as a linear combination
$$[\cc_{g,d}^r(\alpha)]=\mu\cdot [\ww_g]+\nu\cdot \mathcal{BN}_g,$$
for non-negative rational constants $\mu$ and $\nu$, which are determined in \cite{FT}.


\begin{df}
We say that a pointed curve $[C,p]\in \cC_g$ is \emph{Brill-Noether general}, if for every choice of integers $r,d$ and a corresponding Schubert index $\alpha$ of type $(r,d)$, we have
\[
\dim G^r_d(C,p, \alpha)  =\rho(g,r,d, \alpha)  \quad\quad \mbox{ or } \quad\quad G^r_d(C,p, \alpha) = \emptyset.
\]
In particular, for every linear series $\ell\in G^r_d(C)$, the inequality $\rho(\ell,p)\geq 0$ holds.
\end{df}

If $[C,p]$ is a Brill-Noether general pointed curve, by letting $\alpha$ be the trivial Schubert index, we obtain that $C$ is a Brill-Noether general (unpointed) curve.


\begin{lem}
\label{pointediv}
A pointed curve $[C,p]\in \cC_g$ carries no linear series $\ell$ with $\rho(\ell,p)<0$ if and only if it
does not belong to any locus
$\cC_{g,d}^r(\alpha)$, where $\rho(g,r,d, \alpha)=-1$.
\end{lem}

\begin{proof}
One implication being obvious, assume first there exists a linear series $\ell\in G^r_d(C)$ with $w\bigl(\alpha^{\ell}(p)\bigr)>\rho(g,r,d)\geq 0$.
Then we can find a Schubert index $$\alpha':0\leq \alpha_0'\leq \ldots \leq \alpha_r'\leq d-r$$ with $w(\alpha')=\rho(g,r,d)+1 \leq w\bigl(\alpha^{\ell}(p)\bigr)$, such that $\alpha'\leq \alpha^{\ell}(p)$ (lexicographically). Hence $\rho(g,r,d,\alpha')=-1$ and  $[C,p]\in \cC_{g,d}^r(\alpha')$.
Finally, assume we are in the case when there exists a linear series $\ell\in G^r_d(C)$ with $\rho(g,r,d)<-1$. Then we can find $d'>d$ and a Schubert index $\alpha':0\leq \alpha_0'\leq \ldots \leq \alpha_r'\leq d'-r$ with $\alpha'_i\leq d'-d$ and $w(\alpha')=\rho(g,r,d')+1$. Hence $[C,p]\in \cC_{g,d'}^r(\alpha')$, which finishes the proof.
\end{proof}

\vskip 4pt

We now turn to Du Val surfaces. In what follows, we denote by $\equiv$ linear equivalence of divisors on varieties. Following \cite{ABFS} Proposition 2.3, we recall that a set of nine distinct points $p_1, \ldots, p_9$ in $\PP^2$ is said to be \emph{general} if on the blown-up plane $S':=\mbox{Bl}_{\{p_1, \ldots, p_9\}}(\PP^2)$, every effective divisor
$$
D'\equiv d\ell-\nu_1 E_1-\cdots-\nu_9 E_9
$$
with $\nu_i\geq 0$ and satisfying $D\cdot J'=0$ is necessarily a multiple of $J'$. In particular, if $p_1, \ldots, p_9$ are general points, then the sum $p_1+\cdots+p_9\in J'$ is not torsion.

\begin{rem}
Examples of sets of nine general points in $\PP^2(\mathbb Q)$ are easy to produce, if one starts with a concrete elliptic curve defined over $\mathbb Q$. For instance, it is shown in \cite{ABFS} that the following points lying on the elliptic curve $E: y^2=x^3+17$ are general:
$p_1=(-2,3)$, $p_2=(-1,-4),\ p_3=(2,5), \ p_4=(4,9)$,  $p_5=(52,375)$, $p_6=(5234, 37866)$,  $p_7=(8, -23)$,   $p_8=(43, 282)$,  and  $p_9=\Bigl(\frac{1}{4}, -\frac{33}{8} \Bigr).$
\end{rem}

Recall the definition (\ref{p10}) of the points $p_{10}^{(g)}\in J'$, where $g\geq 1$.

\begin{prop}\label{notorsion}
Assume that the points $p_1, \ldots, p_9$ are general. Then for $k=2,\ldots, g$, the difference $p_{10}^{(k)}-p_{10}^{(k-1)}\in \mathrm{Pic}^0(J')$ is not torsion.
\end{prop}
\begin{proof} Using (\ref{p10}), we obtain that $p_{10}^{(k-1)}-p_{10}^{(k)}= p_1+\cdots+p_9$ (with respect to the group law of $J'$), for each $k\geq 2$. As pointed out, this is not a torsion point on $J'$.
\end{proof}

\vskip 3pt

We now introduce the pointed Du Val pencil in $\cc_g$, which is a lift under the forgetful map $\pi:\cc_g\rightarrow \mm_g$ of the pencil of unpointed curves introduced in Section 4 of \cite{ABFS}.
Recall that $S:=\mbox{Bl}_{p_{10}^{(g)}}(S')$ and we denote by $L_g$ the proper transform of the linear system on $S'$ denoted by the same symbol in the Introduction. The linear system of Du Val curves of genus $g-1$ on $S$, that is,
$$
\Lambda_{g-1}:=\bigl|3(g-1)\ell-(g-1)E_1-\cdots-(g-1)E_8-(g-2)E_9\bigr|
$$
appears as a hyperplane in the $g$-dimensional linear system $L_g$. It consists precisely of the
curves of the form $D+J\in L_g$, where $J\subset S$ denotes the proper transform of $J'$ and $D\in \Lambda_{g-1}$.
Since $J\equiv 3\ell-E_1-\cdots-E_{10}$, note that $D\cdot J=1$.

\vskip 3pt

We now choose a Lefschetz pencil in $L_g$, which has $2g-2=C^2$ base points.
Let $X:=\mbox{Bl}_{2g-2}(S)$ be the blow-up of $S$ at those points. Since $C\cdot E_{10}=1$ for $C\in L_g$, the corresponding fibration $f:X\rightarrow \PP^1$ has a section induced by the proper transform of $E_{10}$ on $X$. This induces a pencil in
the universal curve
$$j:\PP^1\rightarrow \overline{\mathcal{C}}_g.$$

In what follows it will be convenient to use the notation $C_1\cup_p C_2$, for a stable curve consisting of two irreducible components $C_1$ and $C_2$ respectively, meeting transversally at a point $p$.

\begin{prop}\label{numbers}
The intersection numbers of the pointed Du Val pencil with the generators of  $\mathrm{Pic}(\cc_g)$ are as follows:
\begin{align*}
j^*(\lambda) = g, \ j^*(\psi) = 1, \ j^*(\delta_{\mathrm{irr}})= 6(g+1), \ j^*(\delta_{1}) = 1, \ \mbox{ }  j^*(\delta_{i})= 0 \ \mbox{ for } \ i=2,\dots ,g-1.
\end{align*}
\end{prop}
\begin{proof}
One has $j^*(\psi)= - E_{10}^2=1$. The restrictions of the classes $\lambda,\delta_{\mathrm{irr}}$, $\delta_2, \ldots, \delta_{g-2}$ follow from \cite{ABFS} Theorem 4.1 and are copied here for the sake of completeness. There exists precisely one element of the pencil $f$ of the type $D+J$, for some $D\in \Lambda_{g-1}$. Since $E_{10}\cdot J =1$ while $E_{10}\cdot D=0$, the marked point lies on the elliptic component of this singular element. The corresponding pointed stable curve is $\bigl[D\cup_{p_{10}^{(g-1)}} J', \ p_{10}^{(g)}\bigr]\in \cc_g$. Hence $j^*(\delta_1)=1$, and since $\pi^*(\delta_1)=\delta_1+\delta_{g-1}$, it follows that $j^*(\delta_{g-1})=0$.
\end{proof}

By direct computation, using (\ref{wei}) and (\ref{bn1}), it follows that the pencil $j(\PP^1)\subset \cc_{g}$ has intersection number zero with the Brill-Noether class $\mathcal{BN}_g$
as well as with the Weierstrass divisor $\ww_g$, that is,
$$
j^*(\mathcal{BN}_g)=(g+3)g-\frac{g+1}{6} (6g+6)-(g-1)=0, \ \  \ \ \mbox{ and }
$$
$$ j^*( [\overline{\mathcal{W}}_g])=-g+{g+1\choose 2}-{g\choose 2}=0.
$$
Since the class of any pointed Brill-Noether divisor lies in the cone spanned by these classes \cite{MR985853} Theorem 1.2, it follows that the intersection number of $j(\PP^1)$ with the closure of any pointed Brill-Noether divisor 
is zero as well.

\vskip 4pt

We are now in a position to complete the proof of our main result.


\begin{proof}[Proof of Theorem \ref{main1}]
We shall establish by induction on $g$ that the general member of the Du Val pencil satisfies the pointed Brill-Noether Theorem.  For $g=1$, we have that $[C,p]\in \cC_1$ and it is well-known that each smooth pointed elliptic curve is Brill-Noether general, see e.g. \cite{MR910206} Theorem 1.1. Assuming the statement for Du Val curves of genus $g-1$, suppose by contradiction that there exist $r,d\geq 1$ and a Schubert index $\alpha$ such that $\dim G^r_d(C,p,\alpha)>\rho(g,r,d, \alpha)$, for each $C\in L_g$, where $\{p\}=C\cap E_{10}$.

\vskip 3pt
Let $j:\PP^1\rightarrow \cc_g$ be a Lefschetz pencil of Du Val curves on $S$. As explained in Proposition \ref{numbers},  the pencil contains  a unique elliptic tail degeneration  $\bigl[D\cup_{p_{10}^{(g-1)}} J', p_{10}^{(g)}\bigr]$, where $D$ is an element of  $\Lambda_{g-1}$. Then the variety $$\overline{G}^r_d\Bigl(D\cup J', p_{10}^{(g)}, \alpha \Bigr)$$ of limit linear series $\ell=(\ell_D, \ell_{J'})\in G^r_d(D)\times G^r_d(J')$ on
$D\cup_{p_{10}^{(g-1)}} J'$ satisfying the ramification condition $\alpha^{\ell}\bigl(p_{10}^{(g)}\bigr)\geq \alpha$ is of dimension at least $\rho(g,r,d,\alpha)+1$. Note that $[D, p_{10}^{(g-1)}]$ can be assumed to be a general Du Val curve of genus $g-1$, for every curve from $\Lambda_{g-1}$ appears as an elliptic tail degeneration in a genus $g$ Du Val pencil.

\begin{figure}[h]
\centering
 \def\svgwidth{0.4\columnwidth}
 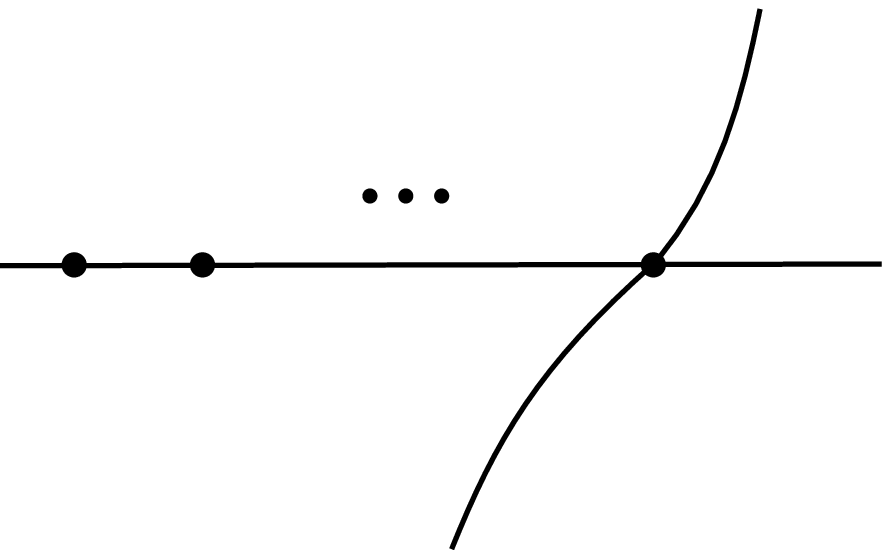
  {\tiny{\caption{The $k$-th step of the elliptic tail specialization in a Du Val pencil.}}}
  \label{fig:cover}
\end{figure}

Let $\ell$ be a general point of an irreducible component $Z$ of $\overline{G}^r_d\Bigl(D\cup J', p_{10}^{(g)}, \alpha \Bigr)$ of maximal dimension, and set $\beta:=\alpha^{\ell_D}(p_{10}^{(g-1)})$. By the additivity of the
Brill-Noether number with respect to marked points, we write
$$
\rho(g,r,d,\alpha)=\rho\Bigl(\ell, p_{10}^{(g)}\Bigr) \geq \rho\Bigl(\ell_D,p_{10}^{(g-1)}\Bigr)+\rho\Bigl(\ell_{J'}, p_{10}^{(g)}, p_{10}^{(g-1)}\Bigr).
$$
By the construction in \cite{MR846932} Theorem 3.3 of the variety of limit linear series, $Z$ is birational to an irreducible component of the product
$$
G^r_d\Bigl(D, p_{10}^{(g-1)}, \beta \Bigr)\times G^r_d\Bigl(J', (p_{10}^{(g-1)}, \beta^{c}), (p_{10}^{(g)}, \alpha)\Bigr).
$$
By assumption, each component of $G^r_d\Bigl(D, p_{10}^{(g-1)}, \beta \Bigr)$ has dimension $\rho(g-1,r,d,\beta)$.

\vskip 3pt

Moving to $J'$, first observe that $\rho\Bigl(\ell_{J'}, p_{10}^{(g)}, p_{10}^{(g-1)}\Bigr)\geq 0$. Indeed, assuming otherwise, we denote by $(a_0, \ldots, a_r)$ and $(b_0, \ldots, b_r)$ the vanishing sequences of $\ell_{J'}$ at the points $p_{10}^{(g-1)}$ and $p_{10}^{(g)}$ respectively, and obtain that there exist indices $0\leq i<j\leq r$ such that
$$
a_i+b_{r-i}=a_j+b_{r-j}=d.
$$
In particular, the underlying line bundle of the linear series $\ell_{J'}$ corresponds to the divisors $a_i\cdot p_{10}^{(g-1)}+b_{r-i}\cdot p_{10}^{(g)}\equiv a_j \cdot p_{10}^{(g-1)}+
b_{r-j}\cdot p_{10}^{(g)}$, from which it follows that $p_{10}^{(g-1)}-p_{10}^{(g)}$ is a torsion class in $\mbox{Pic}^0(J')$, which contradicts Proposition \ref{notorsion}.

\vskip 3pt
Furthermore, it implicitly follows from \cite{MR846932}, and it is spelled-out explicitly in  \cite{MR3296194} Lemma 2.1, that every $2$-pointed elliptic curve $[E, x,y]\in \cM_{1,2}$, where the difference $\OO_E(x-y)$ is not a torsion class, is Brill-Noether general. This follows from the observation that for a line bundle $L\in \mbox{Pic}^d(E)$ which is not given by a divisor on $E$ supported only at $x$ and $y$, the flags in $H^0(E,L)\cong \mathbb C^{d}$ given by the vanishing of sections at $x$ and $y$ respectively, are transversal. In particular, Schubert cycles in $G\bigl(r+1, H^0(E,L)\bigr)$ defined in terms of these flags intersect in the expected dimension. Applying this fact to the case at hand, we find
$$
\dim G^r_d\Bigl(J', (p_{10}^{(g-1)}, \beta^{c}), (p_{10}^{(g)}, \alpha)\Bigr)=\rho(1,r,d, \beta^{c}, \alpha):=\rho(1,r,d)-w(\beta^c)-w(\alpha).
$$

Putting all together, we obtain that
\begin{eqnarray*}
\rho(g,r,d,\alpha)<\dim Z &=& \dim G^r_d\Bigl(D, p_{10}^{(g-1)}, \beta\Bigr)+\dim G^r_d\Bigl(J', (p_{10}^{(g-1)},\beta^c), (p_{10}^{(g)}, \alpha)\Bigr) \\
&=& \rho(g-1,r,d, \beta)+\rho(1,r,d, \beta^{c}, \alpha)\leq \rho(g,r,d, \alpha),
\end{eqnarray*}
which is a contradiction. Therefore, the singular pointed curve $\left(D\cup J', p_{10}^{(g)}\right)$ is Brill-Noether general.
\end{proof}




\begin{cor}\label{ptp} The image of a Du Val pencil $j:\PP^1\rightarrow \cc_g$ is disjoint from all pointed Brill-Noether divisors $\cc_{g,d}^r(\alpha)$.
\end{cor}
\begin{proof}
As noted in Proposition \ref{numbers}, we have $j(\textup{\PP}^1)\cdot \cc_{g,d}^r(\alpha)=0$. Either $j(\PP^1)\cap \cc_{g,d}^r(\alpha)=\emptyset$, or else, $j(\PP^1)\subset
\cc_{g,d}^r(\alpha)$. The proof of Theorem \ref{main1} rules out the second possibility.
\end{proof}

In general it is not known whether $\cC_{g,d}^r(\alpha)$ is pure of codimension $1$. However, when this happens, for instance in the case of the Weierstrass
divisor $\W_g$, Corollary \ref{ptp} shows that \emph{every} pointed Du Val curve is Brill-Noether general with respect to linear series of that type.

\vskip 4pt

\subsection{Towards the effective cone of  \texorpdfstring{$\cc_g$}{ccg}.}

The Slope Conjecture \cite{MR1031904} on effective divisors on $\mm_g$ used to predict that the Brill-Noether divisors $\mm_{g,d}^r$ of curves with a linear series $\mathfrak g^r_d$ where $\rho(g,r, d)=-1$ are extremal. Via Lazarsfeld's result \cite{MR852158}, an  equivalent formulation
of the Slope Conjecture is that the rational curve $R\subset \mm_g$ induced by a Lefschetz pencil of genus $g$ curves  on a general polarized $K3$ surface $(X,H)$, with $H^2=2g-2$ is \emph{nef}, that is,
it intersects every effective divisor on $\mm_g$ non-negatively.
Note that the intersection numbers of $R$ with the generators of $\mbox{Pic}(\mm_g)$ are given as follows, see for instance \cite{MR2123229}:
$$R\cdot \lambda=g+1, \ R\cdot \delta_{\mathrm{irr}}=6g+18 \ \mbox{ and } R\cdot \delta_i=0, \mbox{ for } i=1, \ldots, \left\lfloor \frac{g}{2} \right\rfloor.$$
Although the Slope Conjecture is false for high $g$, see \cite{MR2123229} and \cite{MR2259923}, it is known to hold for $g\leq 9$ and $g=11$. The statement played an important role in Mukai's work on alternative birational models of $\mm_g$ for $g=7, 8, 9$ and has guided the search for geometric divisors $D$ on $\mm_g$ having small slope,  that is, satisfying $R\cdot D<0$, which  necessarily contain the locus in $\mm_g$ of curves that lie on $K3$ surfaces.

\vskip 4pt

It is an interesting question to find an adequate definition of the notion of slope for effective divisors on the universal curve and an analogue of the Slope Conjecture on $\cc_g$.

\begin{prob}\label{kerdes}
 For what values of $g$ is the Du Val pencil $j:\PP^1\rightarrow \cc_g$ nef, that is, $j^*(D)\geq 0$, for every effective divisor $D$ on $\cc_g$? For which $g$ does this inequality hold for all effective divisors $D$ on $\cc_g$ such that $\pi(D)=\mm_g$?
\end{prob}

In light of Corollary \ref{ptp}, a closely related question is whether the Weierstrass divisor $\ww_g$ is extremal in the effective cone $\mbox{Eff}(\cc_g)$. The hypothesis that $\ww_g$ is extremal has recently received further credence due to \cite{Pol}.
Note that for the pull-backs to $\cc_g$ of the effective divisors on $\mm_{6i+10}$  constructed in \cite{MR2259923}, Problem \ref{kerdes} has a negative answer. For instance, when $g=10$, the divisor in question is
\[
\mathcal{Z}_{10}:=\bigl\{[C,p]\in \mathcal{C}_{10}: C \mbox{ lies on a } K3  \mbox{ surface}\bigr\},
\]
and
$[\overline{\mathcal{Z}}_{10}]=7\lambda-5\delta_{\mathrm{irr}} -\delta_1-\delta_9-12\delta_2-12\delta_8-\cdots \in \mbox{Pic}(\cc_{10})$, see \cite{MR2123229} Theorem 1.6. By applying Proposition \ref{numbers}, we compute $j^*([\overline{\mathcal{Z}}_{10}])=-1<0$. We are unaware of any example of an effective divisor $D$ on $\cc_g$ that is \emph{not} a pull-back of an effective divisor from $\mm_g$ and which satisfies $j^*(D)<0$.

\section{Brill-Noether general two-pointed curves via elliptic surfaces}\label{sect2}

In this section we construct explicit smooth $2$-pointed curves of arbitrary genus verifying the Brill-Noether Theorem. Given a smooth curve $C$, distinct points $p,q\in C$ and two Schubert indices $$\alpha:0\leq \alpha_0\leq \ldots \leq \alpha_r\leq d-r \ \ \mbox{  and }\  \ \beta:0\leq \beta_0\leq \ldots \leq \beta_r\leq d-r,$$ we consider the variety $G^r_d\Bigl(C,(p,\alpha), (q,\beta)\Bigr)$ of linear series $\ell\in G^r_d(C)$ verifying ramification conditions at two points:
$$
\alpha^{\ell}(p) \geq \alpha \ \ \mbox{ and } \ \ \alpha^{\ell}(q)\geq \beta.
$$
We say that $[C,p,q]$ satisfies the $2$-pointed Brill-Noether Theorem if for any $\alpha$ and $\beta$,
\[
\dim G^r_d\Bigl(C,(p,\alpha), (q, \beta)\Bigr)=\rho(g,r,d,\alpha, \beta):=\rho(g,r,d)-w(\alpha)-w(\beta),
\]
unless $G^r_d\Bigl(C,(p,\alpha), (q, \beta)\Bigr)= \emptyset$.
Eisenbud and Harris \cite{MR910206} Theorem 1.1 established the $2$-pointed Brill-Noether Theorem for general $2$-pointed curves by use of degeneration. As in the case of $1$-pointed curves, up to now no explicit example of a smooth Brill-Noether general $2$-pointed curve has been known. We construct such curves using decomposable elliptic ruled surfaces.

\vskip 5pt

We start with an elliptic curve $J$ and consider a non-torsion line bundle $\eta\in \mbox{Pic}^0(J)$. Let  $$\phi:Y:=\PP(\OO_J\oplus \eta)\rightarrow J$$ be the ruled surface corresponding  to a \emph{decomposable} rank $2$ vector bundle. We denote by $J_0$ and $J_1$ the disjoint sections of $Y$ such that
$$
N_{J_0/Y}=\eta \mbox{ and } N_{J_1/Y}=\eta^{\vee}.
$$
In particular, $J_0^2=J_1^2=0$. Observe that $J_1\equiv J_0-\phi^*(\eta)$.
We fix a point $r\in J$ and let $f=f_r:=\phi^{-1}(r)$ be the corresponding ruling. For each $g\geq 1$, we denote by $s=s^{(g)}$ the point on the base elliptic curve $J$ determined by
$$
\OO_J(s^{(g)}-r)=\eta^{\otimes g}.
$$
Since $\eta$ is not a torsion class, we have $s^{(g)}\neq r$, for all $g\geq 1$. Furthermore, the difference $s^{(g)} - s^{(g-1)} \in \mbox{Pic}^0(J)$ is not a torsion class. As explained in the Introduction, we set
$$
\{p\}=J_1\cdot f_r \ \mbox{ and } \ \{q^{(g)}\}:=J_0\cdot f_{s^{(g)}}.
$$

\begin{lem}
\label{bpoints}
We have that $h^0(Y, \OO_{Y}(gJ_0+f_r))=g+1$. The general point of the linear system $|gJ_0+f_r|$ is a smooth curve of genus $g$ passing through the points $p$ and $q^{(g)}$.
\end{lem}
\begin{proof}
By direct calculation, using Riemann-Roch, we find that
$$
h^0(Y, \OO_{Y}(gJ_0+f_r))=h^0\Bigl(\OO_J(r)\otimes \mbox{Sym}^g(\OO_J\oplus \eta)\Bigr)=\mbox{deg}\Bigl(\OO_J(r)\otimes \mbox{Sym}^g(\OO_J\oplus \eta)\Bigr)=g+1.
$$
Furthermore, since $K_{Y}\equiv -2J_0+\phi^*(\eta)\equiv -2J_1+\phi^*(\eta^{\vee})$, from the adjunction formula we obtain that a smooth curve $C\in |gJ_0+f_r|$ has genus $g$.

From \cite{FGP}  Proposition 11,  since $\eta$ is non-torsion, the base points of $|gJ_0+f_r|$ lie on $J_0+J_1 = |{-K_Y}|$.
Since $\OO_{J_1}(gJ_0+f_r)=\OO_{J_1}(p)$, the point $p$ must lie in the base locus of $|gJ_0+f_r|$.
Finally, since $\OO_{J_0}(gJ_0+f_r)=\eta^{\otimes g}\otimes \OO_{J_0}(f_r)=\OO_{J_0}(q^{(g)})$, it follows that $q^{(g)}$ belongs to each curve $C\in |gJ_0+f_r|$. Hence, the base locus of $|gJ_0+f_r|$ consists of the points $p$ and~$q^{(g)}$.
\end{proof}

Therefore, on each curve from the linear system $|gJ_0+f_r|$ we can single out two marked points, $p$ and $q=q^{(g)}$. These are precisely the points for which the Brill-Noether Theorem will be proved.

\begin{thm}\label{main4}
The $2$-pointed curve $[C,p,q]\in \cM_{g,2}$, where $C \in |gJ_0+f_r|$ is general and $p$ and $q:=q^{(g)}$ are as above, verifies the $2$-pointed Brill-Noether Theorem, that is,
\[
\dim G^r_d\Bigl(C, (p, \alpha), (q^{},\beta)\Bigr)=\rho(g,r,d,\alpha,\beta)
\quad\quad \mbox{ or } \quad\quad G^r_d\Bigl(C, (p, \alpha), (q^{},\beta)\Bigr) = \emptyset,
\]
for all Schubert indices $\alpha$ and $\beta$.
\end{thm}
\begin{proof}
Assume by contradiction that for a $2$-pointed curve $[C,p,q^{(g)}]$, where $C\in |gJ_0+f|$ is a general element, the Brill-Noether Theorem fails for certain Schubert indices $\alpha$ and $\beta$, that is, there exists a component of $G^r_d\Bigl(C,(p,\alpha), (q, \beta)\Bigr)$ whose dimension exceeds $\rho(g,r,d,\alpha,\beta)$. Then, similarly to the
proof of Theorem \ref{main1}, we consider a specialization of $C$ to the  sublinear system  $\{J_0\}+|(g-1)J_0+f_r|\cong \PP^{g-1}$, which appears as a hyperplane in $|gJ_0+f_r|\cong \PP^g$. The $2$-pointed  curve corresponding to the general element of this subsystem is a curve of the form
$$
[D\cup J_0, \ p\in D,\ q^{(g)}\in J_0]\in \mm_{g,2},
$$
where $D\in |(g-1)J_0+f_r|$ is a smooth curve of genus $g-1$ passing through $p$ and the point $q^{(g-1)}\in J_0\cdot f_{s^{(g-1)}}$. Note that $D\cap J_0=\{q^{(g-1)}\}$. Observe moreover that under the isomorphism $\phi=\phi_{|J_0}:J_0\stackrel{\cong}\rightarrow J$, we have
$$
q^{(g)} - q^{(g-1)}=\phi^*(s^{(g)})-\phi^*(s^{(g-1)})=\phi^*(\eta)\in \mbox{Pic}^0(J_0),
$$
that is, the difference $q^{(g)}-q^{(g-1)}$ is not torsion on $J_0$.

\vskip 4pt

The proof now follows by induction. By semicontinuity, the variety of limit linear series $\ell$ of type $\mathfrak g^r_d$ on $D\cup J_0$ verifying the ramification conditions $\alpha^{\ell}(p)\geq \alpha$ and $\alpha^{\ell}(q^{(g)})\geq \beta$ must have a component $Z$ of dimension strictly greater than $\rho(g,r,d,\alpha,\beta)$. Denote by $\ell=(\ell_D, \ell_{J_0})$ a general point of $Z$. We may assume that $\ell$ is a refined limit linear series. Set $\gamma:=\alpha^{\ell_D}(q^{(g-1)})$. Then $Z$ is birationally isomorphic to the product
$$G^r_d\Bigl(D, (p,\alpha), (q^{(g-1)},\gamma)\Bigr)\times G^r_d\Bigl(J_0,(q^{(g-1)},\gamma^c),(q^{(g)},\beta)\Bigr).$$

By induction on the genus, we may assume that $[D,p,q^{(g-1)}]\in \cM_{g-1,2}$ satisfies the $2$-pointed Brill-Noether Theorem, in particular
$$
\dim G^r_d\Bigl(D, (p,\alpha), (q^{(g-1)},\gamma)\Bigr)=\rho(g-1,r,d, \alpha, \gamma).
$$
Since $q^{(g)}-q^{(g-1)}\in \mbox{Pic}^0(J_0)$ is not torsion, as we have observed $[J_0,q^{(g-1)}, q^{(g)}]\in \cM_{1,2}$ is a Brill-Noether general $2$-pointed curve, hence
$$
\dim G^r_d\Bigl(J_0,(q^{(g-1)},\gamma^c),(q^{(g)},\beta)\Bigr)=\rho(1,r,d,\gamma^c,\beta).
$$

Using the additivity of the Brill-Noether number, we have
$$
\dim  Z=\rho(g-1,r,d,\alpha,\gamma)+\rho(1,r,d,\gamma^c,\beta)=\rho(g,r,d,\alpha,\beta),
$$
a contradiction.
\end{proof}

\begin{rem}
Since a Brill-Noether general $n$-pointed curve supports a Brill-Noether general $m$-pointed curve for all $m<n$ obtained by dropping $n-m$ of the marked points, it follows that the curve $C\in |gJ_0+f_r|$ satisfies the (unpointed) Brill-Noether Theorem as well.
\end{rem}

\begin{rem}
The Du Val curves considered in \cite{ABFS} and in Section \ref{sect1} of this paper are known to lie in the closure in $\cM_g$ of the locus of curves of genus $g$  lying on a $K3$ surface. Algebraic surfaces $\overline{S}\subset \PP^g$ having canonical hyperplane sections have been classified by Epema \cite{MR753823}. All such surfaces are potentially limits in $\PP^g$ of smooth polarized $K3$ surfaces of degree $2g-2$. A criterion for when such surfaces smooth to  $K3$ surfaces is given in \cite{ABS-W} Corollary~26. Du Val surfaces, as well as the decomposable elliptic ruled surfaces considered in Theorem~\ref{main3}, are minimal models of corresponding instances of such objects, see \cite{MR753823}, as well as \cite{ABS-W} Proposition 29. It is natural to ask whether there are explicit examples of Brill-Noether general pointed curves, other than those which are limits of curves on $K3$ surfaces.
\end{rem}

\subsection{Brill-Noether general pointed curves on indecomposable elliptic ruled surfaces}
Here we show how Brill-Noether general one-pointed smooth curves can be constructed also on \emph{indecomposable} elliptic ruled surfaces.
We fix again an elliptic curve $J$ and denote by $\mathcal{E}$ the unique indecomposable vector bundle on $J$ defined by the exact sequence
$$
0\longrightarrow \OO_J\longrightarrow \mathcal{E}\longrightarrow \OO_J\longrightarrow 0.
$$
Let $\varphi\colon X':=\PP(\mathcal{E})\rightarrow J$ be the induced ruled surface. We fix a point $r\in J$ and set $f:=\varphi^{-1}(r)$, therefore $f^2=0$. Let $J_0\subset X'$ be
the unique section of $\varphi$ having $N_{J_0/Y'}=\OO_{J_0}$ and set $\{q\}:=J_0\cdot f_r$.
In a way similar to the proof of Lemma \ref{bpoints}, one can show that the  general element of the linear system $|gJ_0+f|$ is a curve of genus $g$ passing through the point $q$.


Each curve $C\in |gJ_0+f|$ has a distinguished marked point, namely $q\in C\cdot J_0$. In  \cite{MR1218267}, Treibich considers curves in the linear system $|gJ_0+f|$  and sketches an argument using Fay's trisecant formula for showing that a general curve $C\in |gJ_0+f|$ is Brill-Noether general. Reasoning in a way  very similar to  the proof of Theorem \ref{main4}, we prove the stronger fact that the general curve $[C,q]$ satisfies the pointed Brill-Noether Theorem.

\vskip 4pt

Since the linear system $|gJ_0+f|$ cuts out on a general curve $C\in |gJ_0+f|$ the linear system $|\omega_C|+2q$, it follows that $|gJ_0+f|$ has a further base point $q'$, infinitely near to~$q$.\footnote{Added in February 2023. We are grateful to Edoardo Sernesi for pointing out to us this fact, which was overlooked in the published version of this paper.} We denote by $\epsilon\colon X:=\mbox{Bl}_{q, q'}(X')\rightarrow X'$ the blow-up of $X'$ at $q$ and $q'$ and by $E$, respectively $E'$, the corresponding exceptional divisors. We keep denoting by $J_0$ and $f$ the strict transforms of the curves denoted by the same symbols on $X'$.  Finally, let $C\subset X$ be the strict transform of a curve in the linear system $|gJ_0+f|$. Then $C\cdot E'=1$ and $C^2=2g-2$. A Lefschetz pencil in this linear system induces a family of pointed curves
$$\iota\colon \PP^1\rightarrow \cc_g.$$

\begin{prop}
\label{numbers2}
The numerical features of the pencil $\iota\colon \PP^1\rightarrow \cc_g$ are as follows:
$$
\iota^*(\lambda) = g-1, \ \ \iota^*(\psi) = 1, \ \iota^*(\delta_{\mathrm{irr}})= 6(g-1), \ \ \iota^*(\delta_{1}) = 1, \ \ \iota^*(\delta_{g-1})=1,$$
$$  \mbox{ and  }  \iota^*(\delta_{i})= 0 \ \mbox{ for } \ i=2,\dots ,g-2.
$$
\end{prop}

\begin{proof}
We blow-up $X$ at the $2g-2$ base points of a Lefschetz pencil in $|C|$ and denote by $h\colon \widetilde{X}\rightarrow \PP^1$ the induced fibration. Clearly $h^2(\widetilde{X},\OO_{\widetilde X})=0$ and $H^1(\widetilde{X},\OO_{\widetilde X})\cong H^1(J, \OO_J)$ is $1$-dimensional, therefore $\chi(\widetilde{X},\OO_{\widetilde X})=0$. Accordingly,
$$
\iota^*(\lambda)=\chi(\widetilde{X},\OO_{\widetilde{X}})+g-1=g-1.
$$
By the Noether formula, the total number of singular fibres in the pencil $\iota$ is given by
$$
\iota^*(\delta)=c_2(\widetilde{X})+4g-4=12\chi(\widetilde{X},\OO_{\widetilde{X}})-K_{\widetilde{X}}^2+4g-4=6g-4.
$$
In the pencil $\iota$ there exists a unique curve from the linear system $|(g-1)J_0+f|+J_0$, which  is viewed as a hyperplane inside
$|gJ_0+f|$. This singular curve is of the type
\begin{equation}\label{elltail3}
t=\bigl[D\cup E  \cup J_0, \widetilde{q}:=E\cdot E'\bigr]\in \cc_g,
\end{equation}
where $D\in |(g-1)J_0+f|$ is a smooth curve of genus $g-1$ with $D\cap J_0=\emptyset$ (on $X$).
Forgetting the marked point $\widetilde{q}$, the stable model of this curve is $[D\cup_q J_0]\in \mm_g$.
The point $t$ lies on both boundary divisors $\Delta_1$ and $\Delta_{g-1}$, which implies $\iota^*(\delta_1)=\iota^*(\delta_{g-1})=1$, therefore $\iota^*(\delta_{\mathrm{irr}})=6(g-1)$.
\end{proof}

\begin{cor} The numerical features of the pencil  $\bar{\iota}:=\pi\circ \iota:\PP^1\rightarrow \mm_g$ obtained by forgetting the marked point, are given by:
$$
\bar{\iota}^*(\lambda)=g-1, \ \bar{\iota}^*(\delta_{\mathrm{irr}})=6(g-1), \ \bar{\iota}^*(\delta_1)=2, \ \bar{\iota}^*(\delta_i)=0, \mbox{ for } i=2, \ldots, \left\lfloor \frac{g}{2} \right\rfloor.
$$
\end{cor}
\begin{proof} The only thing which has to be observed is that $\bar{\iota}^*(\delta_1)=\iota^*(\delta_1)+\iota^*(\delta_{g-1})=2$.
\end{proof}

Using Proposition \ref{numbers2} it is now immediate to check that the pencil $\iota$, just like the Du Val pencil, satisfies the relations
$$
\iota^*(\mathcal{BN}_g)=0 \ \mbox{ and } \iota^*([\ww_g])=0.
$$

\begin{thm}\label{main2} The general pointed curve $[C, q]$, where $C\in |gJ_0+f|$ and $\{q\}=J_0\cdot f$, verifies the pointed Brill-Noether Theorem.
\end{thm}

\begin{proof}
The proof proceeds by induction on $g$ in a way mirroring the proofs of Theorems \ref{main1} and \ref{main4}. Assume by contradiction that the pointed Brill-Noether Theorem fails for every smooth curve $[C,q]$, where $C\in |gJ_0+f|$. By choosing a Lefschetz pencil $\iota$ in $|gJ_0+f|$ as above, the same conclusion holds for the degenerate pointed curve $t=[D\cup E\cup J_0, \widetilde{q}]$.
That is, the variety of limit linear series $\ell$ of type $\mathfrak{g}^r_d$ on $D\cup E\cup J_0$ such that
$a^\ell(\widetilde{q})\geq \alpha$
has a component $Z$ of dimension strictly greater than $\rho(g,r,d,\alpha)$, for some $r,d,$ and $\alpha$.
For $\ell=(\ell_D, \ell_E, \ell_{J_0})$  a general point of $Z$, let $\gamma_D:=\alpha^{\ell_D}(D\cdot E)$ and $\gamma_{J_0}:=\alpha^{\ell_{J_0}}(J_0 \cdot E)$.
Then $Z$ is birationally isomorphic to
\[
G^r_d\Bigl(D, (D\cdot E, \gamma_D)\Bigr) \times G^r_d\Bigl(E, (E\cdot D, \gamma_D^c), (E\cdot J_0, \gamma_{J_0}^c), (\widetilde{q}, \alpha) \Bigr) \times G^r_d\Bigl(J_0, (J_0\cdot E, \gamma_{J_0})\Bigr).
\]
Both the $3$-pointed rational curve $[E, E\cdot D, E\cdot J_0, \widetilde{q}]$, as well as the $1$-pointed elliptic curve $[J_0, J_0\cdot E]\in \cM_{1,1}$ verify the pointed Brill-Noether Theorem. By induction the same can be assumed for $[D, D\cdot E]\in \mathcal{C}_{g-1}$. 
It follows that
\[
\dim Z = \rho(g-1,r,d,\gamma_D) + \rho(0,r,d,\gamma_D^c,\gamma_{J_0}^c,\alpha) + \rho(1,r,d,\gamma_{J_0})= \rho(g,r,d,\alpha),
\]
a contradiction.
\end{proof}

\bibliographystyle{alphanum}
\bibliography{}

\end{document}